\documentclass[12pt]{article}  

\usepackage{graphicx} 
\usepackage{float}    
\usepackage{verbatim} 
\usepackage{amsmath}  
\usepackage{amssymb}  
\usepackage[colorlinks=true,citecolor=blue,linkcolor=blue,urlcolor=blue]{hyperref}
\usepackage{fullpage}
\usepackage{amsthm,mathtools}
\usepackage{enumerate}
\usepackage{paralist}
\usepackage{xspace}
\usepackage{caption}
\usepackage{bbm}
\usepackage{url}
\usepackage{mathrsfs}

\newtheorem{definition}{Definition}
\newtheorem{assumption}{Assumption}
\newtheorem{theorem}{Theorem}
\numberwithin{theorem}{section}
\newtheorem{lemma}[theorem]{Lemma}

\def\qed{\rule[0pt]{5pt}{5pt}\par\medskip}
\renewcommand{\qedhere}{\hfill ~\qed}
\renewenvironment{proof}{{\noindent\bf Proof.}}{\qedhere}

\newcommand{\minimize}{\mbox{minimize}}

\newcommand{\st}{\mbox{subject to}}

\newcommand{\SFpair}{\left\{\Phix,\Phiu\right\}}
\newcommand{\tf}[1]{\mathbf{#1}}

\newcommand{\Phixh}{\hat{\tf \Phi}_x}
\newcommand{\Phiuh}{\hat{\tf \Phi}_u}
\newcommand{\Phix}{\tf \Phi_x}
\newcommand{\Phiu}{\tf \Phi_u}

\newcommand{\cG}{\mathcal{G}}

\newcommand{\supp}{\textbf{supp}}

\newcommand{\R}{\ensuremath{\mathbb{R}}}
\newcommand{\C}{\ensuremath{\mathbb{C}}}

\newcommand{\hinf}{\mathcal{H}_\infty}
\newcommand{\htwo}{\mathcal{H}_2}
\newcommand{\RHinf}{\mathcal{RH}_\infty}

\newcommand{\Lone}{\mathcal{L}_1}

\newcommand{\s}{\mathcal{S}}

\newcommand{\ttf}[1]{\boldsymbol{#1}}

\usepackage[normalem]{ulem} 
\usepackage{color}

\newcommand{\dA}{\ttf{\Delta}_A}
\newcommand{\dB}{\ttf{\Delta}_B}
\newcommand{\rhoA}{\rho_{\ttf{A}}}
\newcommand{\rhoB}{\rho_{\ttf{B}}}
\newcommand{\epsA}{\epsilon_{\ttf{A}}}
\newcommand{\epsB}{\epsilon_{\ttf{B}}}
\newcommand{\nuA}{\nu_{\ttf{A}}}
\newcommand{\nuB}{\nu_{\ttf{B}}}

\newcommand{\cE}{\mathcal{E}}

\numberwithin{theorem}{section}

\numberwithin{equation}{section}

\title{\LARGE \bf
Sparsity Preserving Discretization With Error Bounds
}

\author{James Anderson, Nikolai Matni, and Yuxiao Chen 
\thanks{J. Anderson is with the Department of Computing + Mathematical Sciences, Y. Chen is  with the Department of Mechanical and Civil Engineering, both at the California Institute of Technology, Pasadena CA, 91125.  {\tt\small james@caltech.edu},  {\tt\small chenyx@caltech.edu} \newline N. Matni is with the Department of Electrical Engineering and Computer Science, UC Berkeley, CA 94709.     \tt\small{nmatni@berkeley.edu}
        }%
}

\begin{document}

\maketitle
\thispagestyle{empty}
\pagestyle{empty}

\begin{abstract}
Typically when designing distributed controllers it is assumed that the state-space model of the plant consists of sparse matrices. However, in the discrete-time setting, if one begins with a continuous-time model, the discretization process annihilates any sparsity in the model. In this work we propose a discretization procedure that maintains the sparsity of the continuous-time model. We show that this discretization out-performs a simple truncation method in terms of its ability to approximate the ``ground truth'' model. Leveraging results from numerical analysis we are also  able to upper-bound the error between the dense discretization and our method. Furthermore, we show that in a robust control setting we can design a distributed controller on the  approximate (sparse) model that stabilizes the dense ground truth model.
\end{abstract}

\section{Introduction}
In this paper we consider the problem of taking a continuous-time model
\begin{align}\label{eq:sys_cont}
\dot{x}(t) = \hat Ax(t) + \hat B_1w(t) + \hat B_2u(t),
\end{align}
where $x(t)\in \R^n, w(t)\in \R^{n_w}, $ and $u(t)\in \R^{n_u}$ are the state, disturbance, and control vectors at time $t$, and discretizing it to take the form
\begin{align}\label{eq:sys_disc}
x_{k+1} =  Ax_k +  B_1w_k +  B_2u_k,
\end{align}
where we assume $k$ is the integer sequence $k=0,1,2,\hdots$ of sampling points and that a zero-order-hold scheme is used.

Discretization, the process of converting~\eqref{eq:sys_cont} to~\eqref{eq:sys_disc}, is a well studied topic and numerous methods have been proposed, however all methods we have encountered destroy the sparsity patterns in $(\hat{A}, \hat{B}_1, \hat{B}_2)$ when constructing $(A,B_1,B_2)$. This is unfortunate as the sparsity patterns typically encode some sort of graph or network structure in the physical system. In this work, motivated by distributed control synthesis, we seek to construct sparse discretizations of~\eqref{eq:sys_cont} that respect the network structure of the continuous-time model and are close (in norm) to the ``true'' discrete models.

\section{Background}

\subsection{Discretization}

The zero-order-hold sampling method~\cite{ChenFrancis} for discretization maps $(\hat{A}, \hat{B}_1, \hat{B}_2)$ to $(A,B_1,B_2)$ by specifying a sample rate $\tau>0$ and setting
\begin{equation}\label{eq:sample}
A = e^{\hat A\tau}, \quad B_i = \int_{0}^{\tau}e^{\hat A\lambda} \mathrm{d} \lambda \hat B_i, \quad i \in \{1,2\},
\end{equation}
where $e^X$ defines the exponential
\begin{equation*}
e^X = \sum_{k=0}^{\infty}\frac{1}{k!}X^k
\end{equation*}
 of a matrix $X \in \mathbb{C}^{n\times n}$. When $\hat{A}$ is non-singular the expression for $B_i$ reduces to $\hat A^{-1}(e^{\hat A\tau}-I)\hat{B}_i$. Define $B = [B_1, B_2]$, then a simple method for computing $(A,B)$ (which is applicable when  $\hat A$ is non-singular) was derived by Van Loan~\cite{Van78} and proceeds as follows; Define the matrix
\begin{equation*}
\Psi = \left[ \begin{array}{cccc} -\hat A^T & I & 0 & 0 \\ 0 & -\hat A^T & Q & 0 \\ 0 & 0 & \hat A & \hat B \\ 0 & 0 & 0 & 0 \end{array}
\right]
\end{equation*}
where $Q=Q^T$ is an $n \times n$ real matrix and compute the exponential $e^{\Psi \tau}$~\cite{MolV03}. The exponential takes the form
\begin{equation*}
e^{\Psi \tau} = \left[\begin{array}{cccc} F_1(\tau) & G_1(\tau) & \star & \star \\  0 & F_2(\tau) & G_2(\tau) & \star \\ 0 & 0 & F_3(\tau) & G_3(\tau) \\ 0 & 0 & 0 & F_4(\tau)   \end{array}
\right],
\end{equation*}
where $F_3(\tau) = e^{\hat A\tau}$ and $G_3(\tau) = \int_0^{\tau}e^{\hat A(\tau - \lambda)}\mathrm{d} \lambda \hat B $. Elements denoted with by $\star$ have analytic expressions but are not required here, the remaining $F_i$ and $G_i$ functions are structurally similar to the case of $i=3$ given above. The reader is referred to~\cite{Van78} for the full details.

An alternative method to the sample-and-hold approach is to take a bi-linear transformation (often referred to as Tustin's method) \cite{ChenFrancis,JiaW14}. In this case $(\hat A, \hat B)$ are mapped to $(A,B)$ by
\begin{equation*}
A = \left(I - \frac{\tau}{2}\hat A \right)^{-1}\left( I + \frac{\tau}{2}\hat A \right), ~~ B = \frac{\tau}{2}\left( I - \frac{\tau}{2}\hat A  \right)^{-1} \hat B,
\end{equation*}
provided that the necessary inverse exists.

It should be clear that any non-trivial sparsity patterns in the system matrices in~\eqref{eq:sys_cont} will be lost if~\eqref{eq:sys_disc} is obtained from either~\eqref{eq:sample} or the bilinear-transformation method as both involve computing a matrix exponential or taking an inverse.

In this work we will focus on approximate discretization based on the transformation~\eqref{eq:sample}. Formally, we would like to construct a matrix $F$ such that for a sparse matrix $X$,  $e^X \approx F$ with $\|e^X-F\|< \delta$, where $\delta$ is known a prioiri  and $F$ is sparse.

It should be pointed out that computing approximations of the matrix exponential is a mature topic in numerical analysis. In general these methods look to compute the exponential in a computationally efficient and stable manner rather than preserve any sparsity pattern \cite{MolV03}. Rational approximations seek to replace $e^{X}$ with $p(X)/q(X)$ where $p$ and $q$ are matrix-valued polynomial functions \cite{Hig05}. Spectral methods are perhaps intuitively the simplest methods; Let $X = VDV^{-1}$ where $D$ is a diagonal matrix and $V$ is a matrix formed from the eigenvectors of $X$,  then $e^X = Ve^{D}V^{-1}$. For matrices with dependent eigenvectors, other factorizations can be used. Clearly neither of these methods attempts to produce a sparse exponential. A related problem that arises in numerical solutions of linear ordinary differential equations is that of computing the \emph{action} $v \mapsto e^{At}v$ \cite{alhH11}. Krylov subspace methods~\cite{Saa92} attempt to approximate this action but again it is not clear how one could incorporate sparsity constraints and obtain error bounds using such a method.

\subsection{Distributed Control}
In this section we briefly review the System Level Synthesis (SLS) framework~\cite{WanMD19, SLStutorial} for solving distributed control problems. We consider the state-feedback  problem where the plant $\tf P$ models the state dynamics given by~\eqref{eq:sys_disc} augmented with the error signal $\bar z_k = C_1x_k + D_{11}w_k + D_{12}u_k$. Compactly the system is denoted as
\begin{equation*}
\tf{P}(z) = \left[ \begin{array}{c|cc} A & B_1 & B_2   \\ \hline C_1 & D_{11} & D_{12} \\ I & 0 & 0 \end{array} \right], 
\end{equation*}
which defines the map
\begin{equation*}
\left[\begin{array}{c} \tf {\bar z}(z)\\ \tf y(z) \end{array} \right] = \tf P(z)  \left[\begin{array}{c} \tf w(z)\\ \tf u(z) \end{array} \right]. 
\end{equation*}
We seek to design a controller $\tf u(z) = \tf K(z)\tf{y}(z) = \tf K(z)\tf{x}(z)$. For the remainder of the paper we drop the dependence on $z$ from our notation and simply use bold-face symbols to denote signals in the z-domain.

Unlike classical control synthesis methods which seek to design controllers that minimize the norm of the map from $\tf w$ to $\tf {\bar z}$, SLS controllers work with the closed-loop system response which maps $\ttf{\delta}_x$ to $(\tf x, \tf u)$ where $\ttf{\delta}_x = B_1\tf w$. The system response is described by
\begin{equation*}
\left[\begin{array}{c} \tf x \\ \tf u \end{array}\right] = \left[\begin{array}{c} \Phix \\ \Phiu \end{array}\right]  \ttf{\delta_x},
\end{equation*}
where
\begin{subequations}\label{eq:SFpair}
\begin{align}
\Phix  &= (zI-A-B_2\tf K)^{-1}, \\ \Phiu  &= \tf K (zI-A-B_2\tf K)^{-1}.
\end{align}
\end{subequations}
The following theorem parameterizes all achievable closed-loop system responses and provides a realization of an internally stabilizing controller~\cite{WanMD19}.
\begin{theorem}\label{thm:sf}
Consider the LTI system~\eqref{eq:sys_disc}, evolving under a dynamic state-feedback control policy $\tf u = \tf K \tf x$. The following statements are true:
\begin{enumerate}
\item The affine subspace defined by
\begin{equation}\label{eq:affine_cons}
 \begin{bmatrix} zI - A & -B_2 \end{bmatrix}\begin{bmatrix} \Phix \\ \Phiu \end{bmatrix} = I, \quad \Phix, \Phiu \in \frac{1}{z}\RHinf
\end{equation}
parameterizes all system responses from $\ttf{\delta_x}$ to $(\tf x,\tf u)$ as defined in~\eqref{eq:SFpair}, achievable by an internally stabilizing state feedback controller $\tf K$.
\item For any transfer matrices $\SFpair$ satisfying~\eqref{eq:affine_cons}, the controller $\tf K = \Phiu \Phix^{-1}$ is internally stabilizing and achieves the desired system response~\eqref{eq:SFpair}.
\end{enumerate}
\end{theorem}
The significance of Theorem~\ref{thm:sf} is that~\eqref{eq:affine_cons} provides an \emph{affine characterization} of $\emph{all}$ achievable system responses. In recent work it was shown that if the affine expression in~\eqref{eq:affine_cons} is not satisfied, it is still possible to construct a stabilizing controller based on an \emph{approximate} system response~\cite{MatWA17}.
\begin{theorem}\label{thm:robust}
Let $(\Phixh, \Phiuh, \ttf{\Delta})$ with $\Phixh, \Phixh \in \frac{1}{z}\RHinf$ be a solution to
\begin{equation}
\begin{bmatrix} zI - A & -B_2 \end{bmatrix} \begin{bmatrix} \Phixh \\ \Phiuh \end{bmatrix} = I + \ttf{\Delta}.  \label{eq:near_sf1}
\end{equation}
Then, the controller  $\tf K = \Phiuh \Phixh^{-1}$ internally stabilizes the system $(A,B_2)$ if and only if $(I + \ttf{\Delta})^{-1}$ is stable.  Furthermore, the actual system responses achieved are given by
\begin{equation*}
\begin{bmatrix} \tf x \\ \tf u \end{bmatrix} = \begin{bmatrix} \Phixh \\ \Phiuh \end{bmatrix}(I+ \ttf\Delta)^{-1}\ttf{\delta_x}.
\end{equation*}
\label{thm:robust}
\end{theorem}
Theorem~\ref{thm:robust} forms the basis of what we term a system level synthesis problem, i.e., a mathematical program that returns a distributed optimal controller. The standard SLS problem is 
\begin{align}\label{eq:SLS}
\minimize_{\gamma \in [0,1)}~& \underset{\Phixh, \Phiuh, \ttf \Delta}{\minimize} \quad  g(\Phixh, \Phiuh) \nonumber \\
&\st \quad \quad \Phixh, \Phixh \in \frac{1}{z}\RHinf,    \nonumber  \\
& \quad    \hspace{1.3cm} \quad \|\ttf \Delta\|< \gamma,  \\
& \quad \hspace{1.7cm} \eqref{eq:near_sf1} , \quad  \begin{bmatrix} \Phixh \\ \Phiuh \end{bmatrix} \in \s . \nonumber
\end{align}
The SLS problem~\eqref{eq:SLS} is quasi-convex and can thus be optimized over. The inner problem is convex, and so a simple bisection on $\gamma$ suffices. Note that the choice of norm on $\ttf \Delta$ must be an induced-norm, this is a sufficient condition for $(I+\ttf \Delta)^{-1}$ to be stable. The cost functional $g$ is chosen to be
\begin{equation*}
\left\| \begin{bmatrix} C_1 & D_{12}\end{bmatrix} \begin{bmatrix} \Phixh \\ \Phiuh \end{bmatrix} \right\|^2.
\end{equation*}
The system norms that we will consider are the $\mathcal{H}_2$, $\Lone$, and the $\mathcal{E}_1$-norms. The $\mathcal{L}_1$-norm of a transfer matrix $\tf G(z)$ is the induced $\ell_{\infty} \to \ell_{\infty}$  norm  which can be computed as 
\begin{equation*}
\|\tf G\|_{\mathcal{L}_1}= \max_{1\le i \le m} \sum_{j=1}^n \sum_{k=0}^{\infty} |G_{ij}[k]|.
\end{equation*}
We define the  $\mathcal E_1$-norm of $\tf G(z)$ as $\|\tf G ^T\|_{\mathcal{L}_1}$. These norms are particularly useful in the context of distributed control because they enjoy a separability property. In particular, The $\htwo$- and $\mathcal{E}_1$-norms are column-wise separable, while the $\Lone$-norm is row-wise separable. Broadly speaking, this means that the norm of the system can be exactly evaluated by computing the norm of each of the columns (rows) individually and then summing the result. The reader is referred to~\cite{WanMD18} for further details.

Finally the constraint $\hat{ \ttf \Phi} \in \s$ encodes temporal and spatial locality constraints on $\{\Phixh,\Phiuh\}$. Together these constraints encode a large class of distributed control problems, and the decomposability provides $O(1)$ synthesis complexity relative to the state-dimension $n$.

\section{Results}
\subsection{Bounding $\ttf \Delta$}\label{sec:delta}
The $\ttf \Delta $ block that appears in Theorem~\ref{thm:robust} allows us to formulate robust control problems, i.e. the design  of a controller that stabilizes  the plant over all realizations of an uncertainty set. In particular, we will consider the ground-truth model to be of the form
\begin{equation*}
A = A_n + \dA, \quad B_2 = B_n + \dB,
\end{equation*}
where
$(A_n, B_n)$ represent the nominal system data and $\dA, \dB$, represent perturbation matrices. The robust problem we are interested in is: given upper-bounds on $\|\dA\|, \|\dB\|$, can we synthesize a robustly stabilizing controller from the nominal system matrices? Consider the case where the nominal system satisfies
\begin{align}\label{eq:aff_pert}
\begin{bmatrix} z I - A_n & - B_n \end{bmatrix} \begin{bmatrix} \Phix \\ \Phiu \end{bmatrix} &=  I,
\end{align}
then
\begin{align*}
\begin{bmatrix} z I - A & - B_2 \end{bmatrix} \begin{bmatrix} \Phix \\ \Phiu \end{bmatrix} &= I +
 \underbrace{\begin{bmatrix}   \dA& \dB \end{bmatrix} \begin{bmatrix} \Phix \\ \Phiu \end{bmatrix}}, \\  &\hspace{2.2cm} =\ttf \Delta
\end{align*}
where $\ttf \Delta$ above is as defined in Theorem~\ref{thm:robust}. Dean et al~\cite{DeaMMRT17} show that  when $\|\dA \|_2\le \rhoA, \|\dB\|_2 \le \rhoB$ then a tractable upper-bound for $\|\ttf \Delta\|_{\mathcal{H}_{\infty}}$ is achievable that takes this information into account.
\begin{lemma}[\cite{DeaMMRT17}]\label{lem:hinf} For $\ttf \Delta$ as defined above and with the bounds $\rhoA, \rhoB$,  for any $\alpha \in (0,1)$
\begin{equation*}
\|\ttf \Delta\|_{\hinf} \le \left\| \left[\begin{array}{c} \frac{\rhoA}{\sqrt{\alpha}}\Phix \\ \frac{\rhoB}{\sqrt{1-\alpha}}\Phiu \end{array} \right] \right\|_{\hinf}.
\end{equation*}
\end{lemma}
The upper-bound above is then used in place of the constraint $\|\ttf \Delta\|_{\mathcal{H}_{\infty}}<\gamma$ in the SLS problem~\eqref{eq:SLS}. Our first results describe how to deal with perturbations that are not described in terms of the $2$-norm. The next two Lemmas are are in the same spirit as the bound above, but are for the $\Lone$ and $\cE_1$-norms.

Assume that we have the bounds $\|\dA\|_{\infty} \le \epsA$ $\|\dB\|_{\infty} \le \epsB$.
 \begin{lemma}\label{lem:L1} Given the scalar bounds $\epsA$ and $\epsB$, then for all $\alpha \in (0,1)$
\begin{align*}
\|\ttf \Delta\|_{\Lone} &\le \left\| \left[\begin{array}{c} \frac{\epsA}{\alpha}\Phix \\ \frac{\epsB}{1-\alpha}\Phiu \end{array} \right] \right\|_{\Lone}
=   \max  \left\{  \frac{\epsA}{\alpha}\| \Phix\|_{\Lone},  \frac{\epsB}{1-\alpha}\| \Phiu\|_{\Lone}  \right\}.
\end{align*}
\end{lemma}
\begin{proof}
From the definition of $\tf \Delta$ we have
\begin{align*}
&\|\tf \Delta \|_{\Lone}  = \left\|\left[  \begin{array}{cc} \frac{\alpha}{\epsA}\dA & \frac{(1-\alpha)}{\epsB}\dB \end{array} \right] \left[  \begin{array}{c} \frac{\epsA}{\alpha}\Phix \\ \frac{\epsB}{1-\alpha}\Phiu \end{array}   \right]    \right\|_{\Lone} \\
&\le
\left\|  \left[  \begin{array}{cc} \frac{\alpha}{\epsA}\dA & \frac{(1-\alpha)}{\epsB}\dB \end{array} \right] \right\|_{\Lone}
\left\|\left[  \begin{array}{c} \frac{\epsA}{\alpha}\Phix \\ \frac{\epsB}{1-\alpha}\Phiu \end{array}   \right] \right\|_{\Lone}\nonumber \\
& \le \left(  \left\| \frac{\alpha}{\epsA}\dA \right\|_{\Lone}+ \left\| \frac{(1-\alpha)}{\epsB}\dB \right\|_{\Lone}  \right)\left\|\left[  \begin{array}{c} \frac{\epsA}{\alpha}\Phix \\ \frac{\epsB}{1-\alpha}\Phiu \end{array}   \right] \right\|_{\Lone} \\ &=
\left(  \frac{\alpha}{\epsA} \left\| \dA \right\|_{\Lone}+ \frac{(1-\alpha)}{\epsB} \left\| \dB \right\|_{\Lone}  \right)\left\|\left[  \begin{array}{c} \frac{\epsA}{\alpha}\Phix \\ \frac{\epsB}{1-\alpha}\Phiu \end{array}   \right] \right\|_{\Lone}  \\ & \le
 \left( \frac{\alpha}{\epsA}\epsA + \frac{(1-\alpha)}{\epsB}\epsB  \right) \left\|\left[  \begin{array}{c} \frac{\epsA}{\alpha}\Phix \\ \frac{\epsB}{1-\alpha}\Phiu \end{array}   \right] \right\|_{\Lone} \\& =
 \left\|\left[  \begin{array}{c} \frac{\epsA}{\alpha}\Phix \\ \frac{\epsB}{1-\alpha}\Phiu \end{array}   \right] \right\|_{\Lone}\\
 & = \max  \left\{  \frac{\epsA}{\alpha}\| \Phix\|_{\Lone},  \frac{\epsB}{1-\alpha}\| \Phiu\|_{\Lone}  \right\}.
\end{align*}
 The first inequality results from applying the triangle inequality, the second comes from the fact that for $ B \in \C^{m \times n_1}$ and $C \in \C^{m\times n_2}$
\begin{equation*}
\|[B ~ C ]\|_{\infty} \le \| B \|_{\infty}+\| C\|_{\infty},
\end{equation*}
(and noting that $\Lone$-norm of a constant matrix is simply the standard matrix induced $\infty$-norm.) and the final inequality comes from substituting in the upper bounds. The last equality follows since
 \begin{equation*}
  \left\| \left[ \begin{array}{c}  X \\  Y \end{array}\right]\right\|_{\infty}= \max \left\{ \| X \|_{\infty}, \| Y \|_{\infty}\right\}.
  \end{equation*}
\end{proof}

Now, assume instead that we have the bounds    $\|\dA\|_{1} \le \nuA$, $\|\dB\|_{1} \le \nuB$.
\begin{lemma}\label{lem:E1}Given the scalar bounds $\nu_{\tf{A}},\nu_{\tf{B}}$, the following bound holds:
\begin{equation*}
\|\ttf \Delta\|_{\cE_1} \le  \left\| \left[ \begin{array}{c} \nu_{\tf{A}} \Phix \\ \nu_{\tf{B}} \Phiu \end{array} \right] \right\|_{\cE_1}.
\end{equation*}
\end{lemma}
\begin{proof}
Follows similar arguments to the proof of Lemma~\ref{lem:L1}.
\end{proof}
Thus Lemmas~\ref{lem:hinf}--\ref{lem:E1} provide upper-bounds that incorporate the perturbation magnitude (in terms of three different norms of $\dA, \dB$) into the SLS problem~\eqref{eq:SLS}. Furthermore, the $\Lone$ and $\mathcal{E}_1$-norm bounds derived here are row- and column-wise separable making them immediately useful for distributed synthesis~\cite{WanMD18}.

\subsection{Sparse Approximation}
From the definition of the matrix exponential, the most obvious way of constructing a sparse approximation is via truncation. Given a sample parameter $\tau>0$, truncating the exponential after  $k=1$ terms we have
\begin{equation*}
A^{trunc}_{\tau} := I+A\tau \approx e^{\hat A \tau}.
\end{equation*}
Clearly $A^{trunc}_{\tau}$ has the same sparsity as $\hat A$. Indeed, this approximation is the basis for first-order Euler methods for solving initial point problems. It is well known that truncation methods do not preserve stability. In the language of Section~\ref{sec:delta}, $A^{trunc}_{\tau} = A_n.$

One advantage of this approximation method is that there is a clean bound for the error.
\begin{theorem}[\cite{MolV03}]\label{thm:trunc}
Given a matrix ${\hat{ A}} \in \mathbb{C}^{n \times n}$ and a constant $\tau > 0$, then
\begin{equation*}
\| A_{\tau}^{trunc}-e^{\hat{A} \tau}\|_2 \le \left( \frac{\|\hat{A}\|_2^2\tau^2}{2} \right) \left(\frac{1}{1- \frac{\tau}{3}\|\hat{A}\|_2} \right).
\end{equation*}
\end{theorem}
This upper-bound would take the form of $\rhoA$ in our robust control problem.

In practice, the bound Theorem~\ref{thm:trunc} is sharp for small $\tau$ and useless for large values. We now propose a second more accurate method for computing an estimate of the matrix exponential. The method is simple; compute the full matrix exponential and then project it onto the support of $\hat A$. Formally, we define support matrix  $H$ as
\begin{equation*}
[\supp( H)]_{ij} = \left\{ \begin{array}{cl} 1 & \text{if } H_{ij}\neq 0 \\ 0 & \text{otherwise}\end{array}.  \right.
\end{equation*}
For a fixed constant $\tau $, the \emph{projected exponential} is given by
\begin{equation*}
A^{proj}_{\tau}:= \supp (|\hat{A}|+|I|)\circ e^{\hat A \tau} \approx e^{\hat{ A}\tau}
\end{equation*}

where $\circ$ denotes the Hadamard (element-wise) product between two matrices of equal dimension. Note that the $\supp$ operation binds before the Hadamard product. We use the notation $| \cdot |$ to denote the element-wise absolute value when applied to matrix.

For the projection of $\hat B_i$, we use $\supp(|\hat{A}|+|I|)$ to approximate $\supp (\int_{0}^{\tau}e^{\hat A\lambda} \mathrm{d} \lambda)$ and have:
\begin{equation*}
  B^{proj}_{i,\tau}:=\supp ((|\hat{A}|+|I|)\cdot |\hat{B}_i|)\circ \int_{0}^{\tau}e^{\hat A\lambda} \mathrm{d} \lambda \hat{B}_i.
\end{equation*}
 In order to obtain bounds on the approximation error we will need to impose some structure on the continuous drift-matrix ${ \hat{A}}$.
\begin{definition}\label{def:band}
Given a matrix $Y \in \C^{n\times n}$, the bandwidth of $ Y$ is the smallest integer $s$ such that $ Y_{jk}=0$ for all $|j-k|\ge s+1$. We use $Y_{(s)}$ to denote that $Y$ has has bandwidth $s$.
\end{definition}
According to Definition~\ref{def:band} a diagonal matrix has a bandwidth of zero, a tridiagonal matrix has bandwidth $s=1$, etc.

Let $E_{ij} = e_ie_j^T$ where $e_i$ is the standard $i^{th}$ basis vector for $\R^n$. Then a bandwidth $s$ matrix can be extracted from a dense matrix via $Y_{(s)} = \sum_{|i-j|\le s}E_{ii}YE_{jj}$, where the summation is taken over all pairs $\{i,j\}$ that satisfy $|i-j|\le s$. 
\begin{assumption}\label{as:band}
The $n\times n$ drift matrix $\hat{A}$ from $\eqref{eq:sys_cont}$ is a banded matrix, or, there exists a permutation matrix $\Pi$ such that $\Pi \hat{A} \Pi^{-1}$ is banded.
\end{assumption}

\begin{figure*}[!tb]
\centering
\includegraphics[scale=.35]{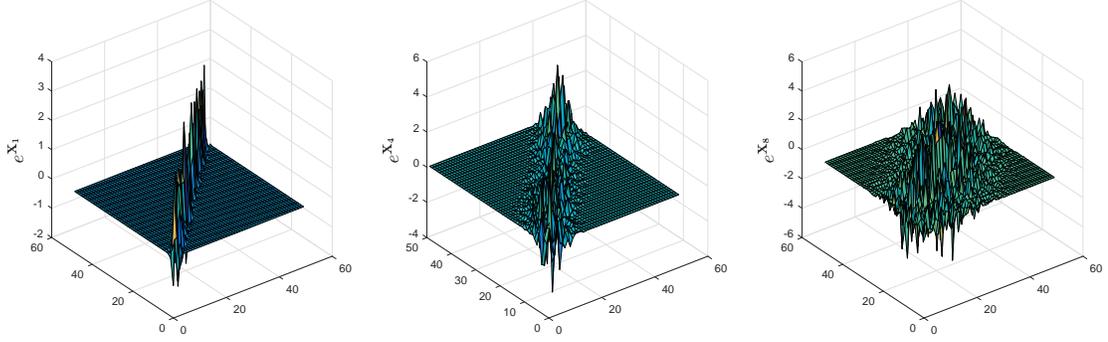}
\caption{The matrix exponential of three banded matrices ${X_1}, {X_4},$ and ${X_8}$ with bandwidth $1,4$, and $8$ are shown from left to right respectively.}
\label{fig:exp_mesh_band}
\end{figure*}

We note that this is not a major assumption, many applications produce matrices naturally in this form, examples can be found in the online catalogue described in~\cite{FLORIDA} and in the recent work~\cite{Vo17}. While the exponential of a banded matrix is formally dense, numerical analysts have noted that elements away from the diagonal decay rapidly. Specifically, ``\emph{provided that $A$ is a banded matrix, $e^{A}$ is itself within an exceedingly small distance from a banded matrix.}''\cite{Ise00}. This observation is illustrated qualitatively in Figure~\ref{fig:exp_mesh_band}.

Recall the (non sub-multiplicative) max-norm of a matrix $X$ defined as
\begin{equation*}
\|X\|_{\text{max}} := \max_{k,l} \quad |X_{kl}|.
\end{equation*}
The following result provides a bound on the elements of the matrix exponential obtained from a banded matrix.
\begin{theorem}[\cite{Ise00}]\label{thm:Iserles}
Let $A = e^{\hat{ A}\tau}$ for some $\tau >0$, where $\hat{A}$ is a banded matrix with bandwidth $s\ge 1$. Let $\alpha = \|\hat{A}\tau \|_{\text{max}}$.Then for $|i-j| \gg 1$,
\begin{align*}
|A_{ij}| &\le \underbrace{\left(\frac{\alpha s}{|i-j|} \right)^{\frac{|i-j|}{s}}\left[e^{\frac{|i-j|}{s}}- \sum_{m=0}^{|i-j|-1}\frac{(|i-j|/s)^m}{m!}\right]} .\\ & \hspace{3.8cm}\mathcal{B}_{ij}( \alpha,s)
\end{align*}
\end{theorem}
Under the assumption that $\hat{A}$ from~\eqref{eq:sys_cont} is banded (with bandwidth $s$), we define $A = A_n + \dA$ where $A_n = A_{\tau}^{proj}$, thus $A_n$ is banded and $\dA$ is is the complement of a banded matrix. Appealing to theorem~\ref{thm:Iserles} we can derive upper-bounds on $\|\dA\|$ for various norms. Moreover, the upper-bounds can be computed using only scalar operations. For any choice of norm it follows that
\begin{align*}
\|\dA\| = \left\|\sum_{|i-j|>s} E_{ii}e^{\hat{A}\tau}E_{jj} \right\| &\le \sum_{|i-j|>s} \left\| E_{ii}e^{\hat{A}\tau}E_{jj} \right\| \\
& = \sum_{|i-j|>s} |[e^{\hat{A}\tau}]_{ij}| \\
& \le \sum_{|i-j|>s} \mathcal{B}_{ij}(\alpha, s).
\end{align*}
Define $\rhoA^{\star}, \epsA^{\star}, $ and $\nuA^{\star}$ to be $\|\dA\|$ for the $2,\infty,$ and $1$-norm respectively.
\begin{theorem}\label{thm:Abounds}
Let $\rhoA^{\star}, \epsA^{\star}, $ and $\nuA^{\star}$ be defined as above. Given a matrix $\hat A \in \R^{n\times n}$ of bandwidth $s$. Then $\rhoA^{\star}, \epsA^{\star}, $ and $\nuA^{\star}$ can be upper-bounded by computations that involve only scalar operations on the indices $i,j$ and parameters $\alpha$ and $s$. Specifically,
\begin{equation*}
\rhoA^{\star} \le \sum_{|i-j|>s} \mathcal{B}_{ij}(\alpha, s), ~~ \epsA^{\star} \le \max_{1 \le i \le n} \sum_{\tiny{\begin{array}{c}j=1\\ |i-j|>s \end{array}}}^n  \mathcal{B}_{ij}(\alpha, s),
\end{equation*}
\begin{equation*}
\text{and } \nuA^{\star} \le  \max_{1 \le j \le n} \sum_{\tiny{\begin{array}{c}i=1\\ |i-j|>s \end{array}}}^n  \mathcal{B}_{ij}(\alpha, s).
\end{equation*}
\end{theorem}
\begin{proof}
The proof follows by applying Theorem~\ref{thm:Iserles}, the inequalities derived above, and the definition of the $1$ and $\infty$ norms.
\end{proof}
Similar bounds for the various norms  of $\|\dB\|$ are easily derived.
In Figure~\ref{fig:boundtest} the upper-bound on $\|\dA\|_2$ from Theorem~\ref{thm:Abounds} is compared to the true value for matrices of bandwidth $4$, i.e. $\hat{A}_{(4)}\in \R^{n\times n}$. We show how the upper-bound changes as a function of the dimension of the matrix where $n\in \{20,40,60,100, 200,500,1000\}$. It can be clearly seen, that even for large matrices, the estimates $\epsA$ are easily small enough to be useful. Upper-bounds on the two other norms of  $\dA$ and $\dB$ give very similar results and are omitted due to space constraints.
\begin{figure}[!htb]
\centering
\includegraphics[scale=.45]{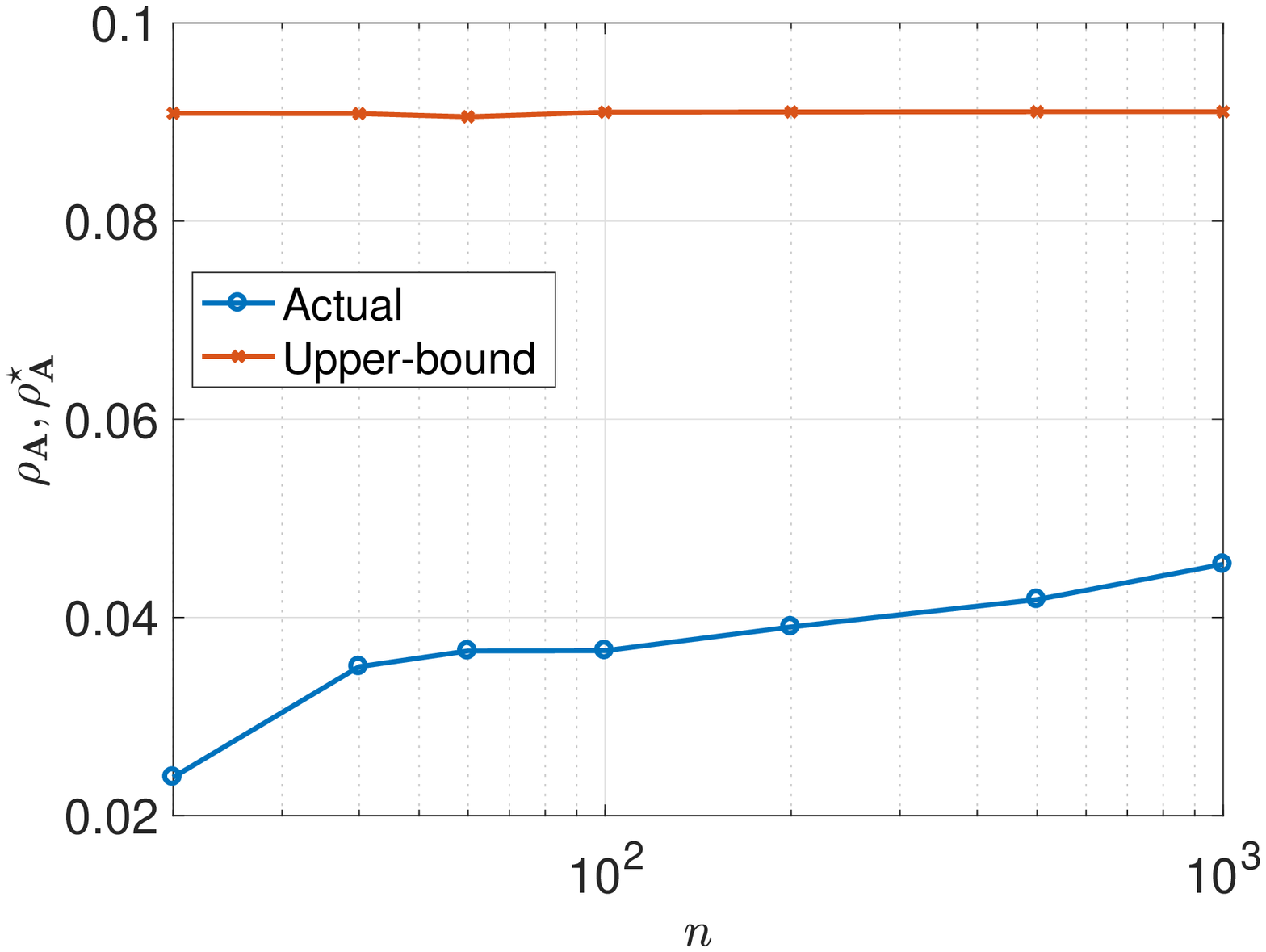}
\caption{The actual error magnitude $\|\dA\|_2 = \epsA^{\star}$ and the upper-bound for $\|\dA\|_2$ given by $\epsA$ from Theorem~\ref{thm:Abounds}.}
\label{fig:boundtest}
\end{figure}

\section{Examples}
We demonstrate the proposed method on a power grid control example. The model comes from \cite{zimmerman2011matpower} and has 57 buses with 7 generator buses. The network topology is   is shown in Figure \ref{fig:power_grid_network}.
\begin{figure}
  \centering
   \includegraphics[width=0.5\columnwidth]{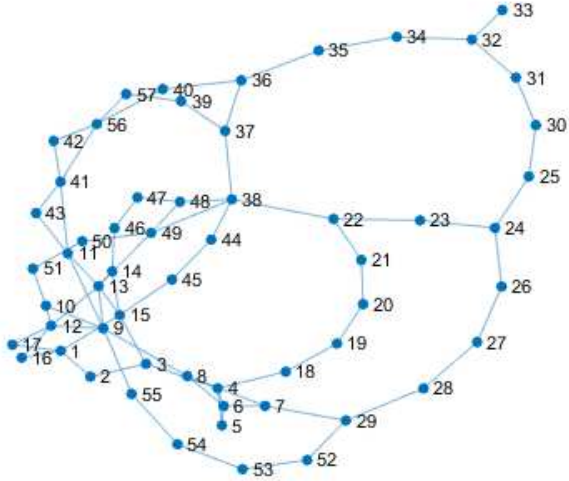}
  \caption{Topology of the power network with dynamics given by~\eqref{eq:grid_model}.}
   \label{fig:power_grid_network}
 \end{figure}
The power grid dynamics are described as follows:
\begin{equation}\label{eq:grid_model}
\begin{aligned}
  \dot{\theta}_i&=\omega_i,\\
  M_i\dot{\omega}_i&=-D_i \omega_i - d_i-u_i -\sum\limits_{j\in\mathcal{N}_i}H_{ij}(\theta_i-\theta_j), i\in\mathcal{G}\\
  0&=-D_i \omega_i - d_i-u_i -\sum\limits_{j\in\mathcal{N}_i}{H_{ij}(\theta_i-\theta_j)}, i\in\mathcal{L},
\end{aligned}
\end{equation}
where $\theta_i$ and $\omega_i$ are the phase angle and frequency of the voltage at bus $i$, $d_i$ is the uncontrollable load at bus $i$ which is treated as a disturbance. $u_i$ is the controllable load, which is used to regulate bus $i$. $\mathcal{G}$ and $\mathcal{L}$ represent the set of generator buses and the set of pure load buses. In this example $\cG = \{1   , 2,  3,  6,  8,  9,  12\}$ and $\mathcal{L} = \{1,\hdots, 57\} \setminus \mathcal G$. For a generator bus, $M_i$ is the inertia and $D_i$ is the damping coefficient; for a load bus, there is zero inertia and $\omega_i$ is determined by an algebraic equation. A generator bus is modeled with 2 states ($x_i=[\theta_i,\omega_i]^T$); and a load bus is modeled with 1 state ($x_i=\theta_i$). $H_{ij}$ represents the sensitivity of the power flow to phase variations, it is nonzero when bus $i$ and bus $j$ are neighbors\footnote{We use the parameter $H_{ij}$ in this paper instead of the more common $B_{ij}$ to avoid confusion with the system matrix $B$.}. In our example we assume an impulsive disturbance hits  bus 3 and affects its frequency. 

\begin{figure}[!t]
  \centering
  \includegraphics[width=0.65\columnwidth]{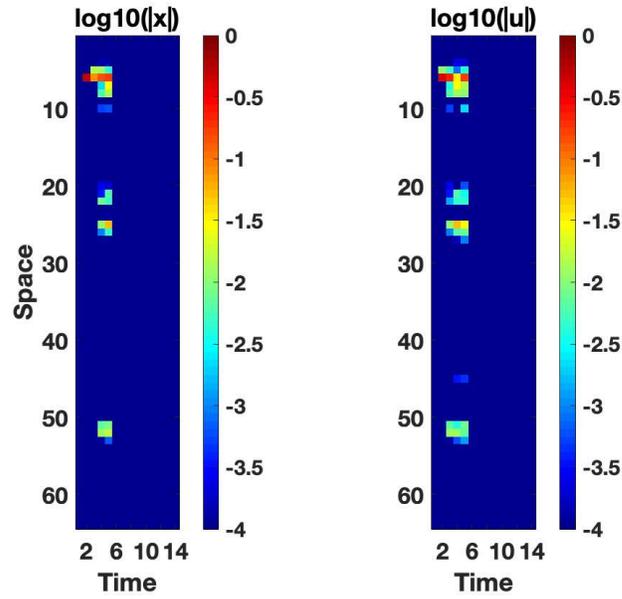}
  \caption{The localized system response implemented on the sparse model. Clearly the response to the disturbance is localized in time and space.}
  \label{fig:sparse_model}
\end{figure}

Using the sparse discretized model, the distributed controller derived from the SLS problem~\eqref{eq:SLS} is feasible for an FIR horizon $T\ge 5$ and the locality radius of $d\ge4$.\footnote{The parameter $T=5$ imposes that after 5 time-steps the disturbance has no affect on the state and that control action is no longer necessary. The locality radius of $4$ ensures that once a disturbance has hit a node in the network, only subsystems within $4$ hops of the disturbance feel the effect. Likewise, only controllers in this region need act. Technical details can be found in~\cite{WanMD18} and case studies in~\cite{SLStutorial}.} In this example we have that $\|\dA\|  = \{0.455, 0.394, 0.873\} $ and $\|\dB\|= \{0.001, 0.006, 0.001\}$ for the $1,2,$ and $\infty$ norms respectively (to 3 s.f.).

When using the exact discretized model obtained using~\eqref{eq:sample}, the SLS synthesis is only feasible if all nodes in the network respond to any disturbance hitting the network, which requires $d=12$ in the example case. This means that distributed and localized control is not possible (i.e. the SLS problem has no localization and is thus a centralized controller) if the underlying model~\eqref{eq:grid_model} is converted to discrete time using standard methods. In Figure~\ref{fig:sparse_model} the closed loop response of the controller on the sparse model is plotted. In Figure~\ref{fig:dense_model}, we show the system response when the controller is designed on the sparse model and implemented on the dense discrete model. Despite the model mismatch the robust controller still manages to localize the disturbance in time and space.

\begin{figure}[!tb]
  \centering
  \includegraphics[width=0.65\columnwidth]{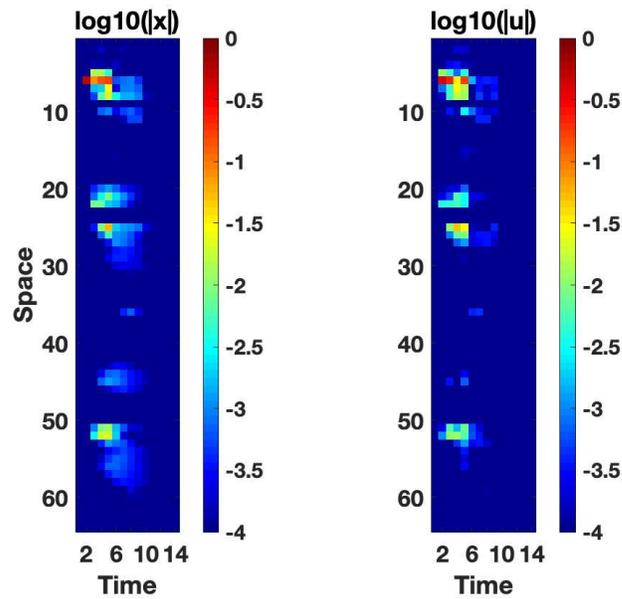}
  \caption{The controller is designed on the sparse model and simulated on the dense model from~\eqref{eq:sample}.}
  \label{fig:dense_model}
\end{figure}

\section{Conclusion}
We have presented a simple projection based method for sparsity-preserving discretization of a continuous-time dynamical system. For the special case of banded matrices, bounds on the approximation error were derived and shown to perform well in practice. For non-banded systems, an existing bound exists which is useful in certain sampling parameter regimes. These bounds were then incorporated into the $\ttf \Delta$ uncertainty parameter in the SLS framework for distributed control. The results were then illustrated on a 57-bus power network where it was shown that an SLS distributed controller designed on the sparse approximation performs well when implemented on the dense ``ground truth'' model. 

Future work will involve looking at how we can derive bounds that are applicable for non-banded systems that can be applied when Theorem~\ref{thm:trunc} is not applicable. When fast but inexact approximations of the matrix exponential are used, the ideas in this paper are also applicable and will be examined. It would also be interesting to see how these results can be carried over to classical robust control problems.

\section*{Acknowledgements}
J. Anderson and Y. Chen are supported by PNNL on grant 424858. J. Anderson is additionally supported by NSF grants  CCF 1637598, ECCS 1619352, and by ARPA-E through the GRID DATA program. N. Matni is generously supported in part by ONR awards N00014-17-1-2191 and N00014-18-1-2833 and the DARPA Assured Autonomy (FA8750-18-C-0101) and Lagrange (W911NF-16-1-0552) programs.
\bibliographystyle{IEEEtran}
\bibliography{biblio}

\end{document}